\documentclass[final]{siamart1116}

\usepackage{amsfonts}
\usepackage{amsopn}
\usepackage{color}

\numberwithin{theorem}{section}

\newcommand*{\PREFIX}{FREMDER}
\newcommand*{\Prefix}{Fremder}
\newcommand*{\prefix}{fremder}
\newcommand{\TheTitle}{On {\PREFIX}vectors: Vectors Orthogonal to Their Images Under Linear Transformations}

\headers{{\Prefix}vectors}{M. G. Reuter}

\title{{\TheTitle}\thanks{Submitted to the editors \today.
\funding{This work was supported by startup funds from the Institute for Advanced Computational Science at Stony Brook University.}}}

\author{
Matthew G. Reuter\thanks{Department of Applied Mathematics and Statistics \& Institute for Advanced Computational Science, Stony Brook University, Stony Brook, NY 11794 (\email{matthew.reuter@stonybrook.edu})}
}

\begin{document}

\maketitle

\date{\today}

\begin{abstract}
Geometrically, the eigenvectors of a square matrix $\mathbf{A}$ are not rotated by $\mathbf{A}$. Here we consider vectors that are rotated $\pi/2$ by $\mathbf{A}$; that is, vectors orthogonal to their images. We call these vectors {\prefix}vectors of $\mathbf{A}$ and discuss conditions for their existence. We also define {\prefix}values, scalars $z$ such that $z\mathbf{I}-\mathbf{A}$ has a {\prefix}vector, and discuss several known applications for {\prefix}vectors.
\end{abstract}



\section{Introduction}
\label{sec:intro}
Determining the eigenvalues and eigenvectors of a square matrix (or linear transformation, more generally) has become a standard topic of introductory linear algebra courses \cite{bk:strang-2006}. Algebraically, the eigenvalues and eigenvectors satisfy the canonical equation
\[
\mathbf{A}x = \lambda x,
\]
where $\mathbf{A}\in\mathbb{C}^{n\times n}$ is the matrix, $\lambda\in\mathbb{C}$ is one of its eigenvalues, and $x\in\mathbb{C}^n$ is an associated eigenvector. Geometrically, an eigenvector can be interpreted as a vector that is not rotated by $\mathbf{A}$; we make no distinction between ``parallel'' and ``anti-parallel'' directions.

Since the 1960s, Gustafson and co-workers \cite{bk:gustafson-2012} have more deeply examined the rotation of vectors by $\mathbf{A}$ and asked the related question, what vector is rotated the most by $\mathbf{A}$? Such a vector is called an antieigenvector of $\mathbf{A}$, and its antieigenvalue is the angle of rotation. As discussed in \cite{bk:gustafson-2012}, antieigenvalues lead to a matrix trigonometry and have found numerous applications in, \textit{e.g.}, numerical analysis and finance. Additionally, $\mathbf{A}$'s antieigenvectors have a close relationship with its eigenvectors when $\mathbf{A}$ is Hermitian and positive definite.

In this work we go one step beyond antieigenvector analysis and consider the properties that $\mathbf{A}$ must satisfy for it to have an antieigenvalue of $\pi/2$. That is, which $\mathbf{A}$ rotate a vector to be orthogonal to its image? We thus seek to find $\mathbf{A}$ and $x$ satisfying
\begin{equation}
\left< x, \mathbf{A}x \right> = 0,
\label{eq:fremdervector}
\end{equation}
where we assume the usual inner product for complex vectors. As a play on words, we call such $x$ {\prefix}vectors. The word \emph{eigen} means ``own'' in German (the image vector belongs to its eigenvector's one-dimensional subspace), whereas \emph{\prefix} means ``stranger'' (the {\prefix}vector and its image share no common component).

Properties of the pencil $z\mathbf{I}-\mathbf{A}$ help develop the notion of a {\prefix}value. Recall that eigenvalues can alternatively be defined as values $z$ that make $z\mathbf{I}-\mathbf{A}$ singular. We define a {\prefix}value of $\mathbf{A}$ as $z$ such that $z\mathbf{I}-\mathbf{A}$ has an antieigenvalue of $\pi/2$.

The remainder of this note discusses properties that $\mathbf{A}$ must satisfy to have {\prefix}vectors and for $z$ to be a {\prefix}value. \Cref{sec:applications} begins by motivating {\prefix}vectors with a brief, and likely incomplete, overview of existing applications. We then summarize some pertinent aspects of matrix theory in \cref{sec:prelims}. Results on the existence of {\prefix}vectors and {\prefix}values for several cases are presented in \cref{sec:fremdervectors}. We finally conclude in \cref{sec:calculating} with a short discussion on calculating {\prefix}values and {\prefix}vectors.

\section{Known Applications}
\label{sec:applications}
{\Prefix}vectors and {\prefix}values are fundamentally interesting quantities that also have applications. Here we list a few, knowing that there are probably many others of which we are unaware. These applications do not typically consider the geometric interpretation of {\prefix}vectors, but use them nevertheless.

One application is in computer graphics. Many solid bodies of interest can be represented by quadric surfaces, and it may be desirable to find the intersection of two such surfaces \cite{levin-73-1979}. As we will discuss in \cref{sec:calculating}, this problem can be recast to use {\prefix}vectors.

Our idea for {\prefix}vectors originated in our research on nanotechnology, where we aim to understand how electric current flows through molecules. These systems are sufficiently small that quantum mechanics is required \cite{bk:cuevas-2010} and the system's properties are described by a linear operator (the Hamiltonian). Eigenvectors of the Hamiltonian relate to maxima in the electric conductance. We recently showed that certain minima in the conductance correspond to {\prefix}vectors of the Hamiltonian \cite{sam-ang-053002-2017}.

We know of two other applications for {\prefix}vectors in quantum mechanics. The first is in spectroscopy \cite{bk:cohen-tannoudji-1977}, where transitions between two quantum states are measured. In some cases the quantum state must change in response to the spectroscopic setup; such a quantum state is a {\prefix}vector of the transition matrix. The second is in condensed matter physics, where the existence of surface states in a material can be framed in terms of {\prefix}vectors \cite{dwivedi-134304-2016}.

\section{Preliminaries}
\label{sec:prelims}
Throughout this note we will denote the spectrum of $\mathbf{A}$ by $\lambda(\mathbf{A})$. Several basic results from matrix theory will also be useful in our discussion. We state them here without proof, referencing any standard text on linear algebra (\textit{e.g.}, \cite{bk:strang-2006}) for more information. 

First, $\mathbf{A}$ is normal if it has a complete set of orthogonal eigenvectors. Second, if $\mathbf{A}$ is Hermitian (\textit{i.e.}, self-adjoint, $\mathbf{A}=\mathbf{A}^\dagger$), then it is normal and its eigenvalues are real. In such a case, $\mathbf{A}$ is classified as positive (semi-)definite if all eigenvalues are positive (non-negative), negative (semi-)definite if all eigenvalues are negative (non-positive), or indefinite if $\mathbf{A}$ has both positive and negative eigenvalues. Third, we extend a similar classification scheme to skew-Hermitian matrices, $\mathbf{A}=-\mathbf{A}^\dagger$, which are normal and have purely imaginary eigenvalues. We use the signs of the eigenvalues' imaginary parts to determine positive definiteness, \emph{etc}. Finally, an arbitrary matrix $\mathbf{A}$ can be decomposed into a Hermitian part, $\mathbf{B}=(\mathbf{A}+\mathbf{A}^\dagger)/2$, and a skew-Hermitian part, $\mathbf{C}=(\mathbf{A}-\mathbf{A}^\dagger)/2$, such that $\mathbf{A}=\mathbf{B}+\mathbf{C}$.

\section{The Existence of {\Prefix}vectors and {\Prefix}values}
\label{sec:fremdervectors}
There are several trivial {\prefix}vectors to mention before proceeding. One is $x=0$, which clearly satisfies Eq.\ \cref{eq:fremdervector}. Any $x\in\ker(\mathbf{A})$ also satisfies Eq.\ \eqref{eq:fremdervector}, as does any $x\in\ker(\mathbf{A}^\dagger)$. In all three of these cases, the angle of rotation between $x$ and either $\mathbf{A}x$ or $\mathbf{A}^\dagger x$ is not well-defined, and might best be regarded as 0 because $x$ is an eigenvector of $\mathbf{A}$.

\begin{definition}
Any $x\in\ker(\mathbf{A})\cup\ker(\mathbf{A}^\dagger)$ is a trivial {\prefix}vector of $\mathbf{A}$.
\end{definition}

The remainder of our discussion will focus on nontrivial {\prefix}vectors. \Cref{thm:fremdervector-sum} provides one general result.

\begin{lemma}
Let $x$ be a {\prefix}vector of $\mathbf{A}$ and let $y\in\ker(\mathbf{A})\cap\ker(\mathbf{A}^\dagger)$. Then $x+y$ is also a {\prefix}vector of $\mathbf{A}$.
\label{thm:fremdervector-sum}
\end{lemma}
\begin{proof}
The proof is straightforward using Eq.\ \cref{eq:fremdervector}.
\end{proof}

\subsection{{\Prefix}vectors}
\label{subsec:exist-fremdervectors}
The simplest place to begin is with Hermitian matrices.

\begin{theorem}
Let $\mathbf{A}$ be Hermitian. Then $\mathbf{A}$ has a nontrivial {\prefix}vector if and only if it is indefinite.
\label{thm:hermitian}
\end{theorem}
\begin{proof}
\Cref{thm:fremdervector-sum} implies that the trivial {\prefix}vectors of $\mathbf{A}$ form a subspace. Without loss of generality, let $x$ be a nontrivial {\prefix}vector with no component in this subspace, and let $\lambda_j$ and $\varphi_j$ be the corresponding eigenvalues and normalized eigenvectors of $\mathbf{A}$. Write $x=\sum_j c_j \varphi_j$ such that, from Eq.\ \cref{eq:fremdervector},
\begin{equation}
0 = \left< x, \mathbf{A}x \right> = \sum_{j=1}^n \left| c_j \right|^2 \lambda_j.
\label{eq:hermitian}
\end{equation}
Equation \eqref{eq:hermitian} has nontrivial solutions if and only if $\mathbf{A}$ has at least one negative eigenvalue and at least one positive eigenvalue such that the expansion coefficients can be chosen to cancel the positive and negative terms.
\end{proof}





A similar result holds for skew-Hermitian matrices.

\begin{corollary}
Let $\mathbf{A}$ be skew-Hermitian. Then $\mathbf{A}$ has a nontrivial {\prefix}vector if and only if it is indefinite (that is, $\mathbf{A}$ has at least one eigenvalue with positive imaginary part and at least one eigenvalue with negative imaginary part).
\label{thm:skew-hermitian}
\end{corollary}
\begin{proof}
From Eq.\ \cref{eq:fremdervector}, the {\prefix}vector $x$ must satisfy
\[
0 = \left< x, \mathbf{A}x \right> = \left< x, i\mathbf{A}x \right>.
\]
Because $i\mathbf{A}$ is Hermitian, the result follows from \cref{thm:hermitian}.
\end{proof}

The investigation of nontrivial {\prefix}vectors for arbitrary matrices is very similar to that of Hermitian matrices. The key difference is that we first decompose the matrix into its Hermitian and skew-Hermitian parts.

\begin{theorem}
Let $\mathbf{A}$ have Hermitian part $\mathbf{B}$ and skew-Hermitian part $\mathbf{C}$. If $x$ is a nontrivial {\prefix}vector of $\mathbf{A}$, then the following three conditions must hold.
\begin{enumerate}
\item[(i)] $x$ is a {\prefix}vector of $\mathbf{B}$.
\item[(ii)] $x$ is a {\prefix}vector of $\mathbf{C}$.
\item[(iii)] At least one of $\mathbf{B}$ and $\mathbf{C}$ is indefinite.
\end{enumerate}
\label{thm:arbitrary}
\end{theorem}
\begin{proof}
From Eq.\ \cref{eq:fremdervector} and $\mathbf{A}=\mathbf{B}+\mathbf{C}$,
\[
0 = \left< x, \mathbf{A}x \right> = \left< x, \mathbf{B}x \right> + \left< x, \mathbf{C}x \right>.
\]
Notice that $\left<x, \mathbf{B}x \right>$ is real and $\left< x, \mathbf{C}x \right>$ is purely imaginary. Thus, both terms must be $0$ independently, meaning $x$ is a {\prefix}vector of both $\mathbf{B}$ and $\mathbf{C}$.

We show condition (iii) by contradiction. Suppose that neither $\mathbf{B}$ nor $\mathbf{C}$ is indefinite. From the proof of \cref{thm:hermitian}, $\mathbf{B}$ must then be positive or negative semi-definite such that $x$ is a trivial {\prefix}vector of $\mathbf{B}$. $x$ is similarly a trivial {\prefix}vector of $\mathbf{C}$. Finally, $x\in\ker(\mathbf{B})\cap\ker(\mathbf{C})$ implies $x\in\ker(\mathbf{A})$ such that $x$ is a trivial {\prefix}vector of $\mathbf{A}$. Thus, at least one of $\mathbf{B}$ and $\mathbf{C}$ must be indefinite for $\mathbf{A}$ to possess a nontrivial {\prefix}vector.
\end{proof}

\subsection{{\Prefix}values}
We have so far considered the properties $\mathbf{A}$ must satisfy to have nontrivial {\prefix}vectors, Let us now shift focus to the {\prefix}values of $\mathbf{A}$ and discuss bounds for their existence.

\begin{theorem}
Let $\mathbf{B}$ and $\mathbf{C}$ be the Hermitian and skew-Hermitian parts of $\mathbf{A}$, respectively. If $z$ is a {\prefix}value of $\mathbf{A}$, then the following three conditions hold.
\begin{enumerate}
\item[(i)] $\min \lambda(\mathbf{B}) \le \mathrm{Re}(z) \le \max \lambda(\mathbf{B})$.
\item[(ii)] $\min \mathrm{Im}(\lambda(\mathbf{C})) \le \mathrm{Im}(z) \le \max \mathrm{Im}(\lambda(\mathbf{C}))$.
\item[(iii)] Conditions (i) and (ii) cannot both be strict equalities. For instance, if $\mathrm{Re}(z)=\min\lambda(\mathbf{B})$, then $\min \mathrm{Im}(\lambda(\mathbf{C})) < \mathrm{Im}(z) < \max \mathrm{Im}(\lambda(\mathbf{C}))$.
\end{enumerate}
\label{thm:arbitrary-value}
\end{theorem}
\begin{proof}
By definition, $z$ is a {\prefix}value of $\mathbf{A}$ if $z\mathbf{I}-\mathbf{A}$ has a nontrivial {\prefix}vector. Invoking \cref{thm:arbitrary}, we thus need to consider the definiteness of the Hermitian and skew-Hermitian parts of $z\mathbf{I} - \mathbf{A}$.

$\mathrm{Re}(z)\mathbf{I} - \mathbf{B}$ is the Hermitian part of $z\mathbf{I}-\mathbf{A}$, and cannot be positive or negative definite by \cref{thm:arbitrary}. Thus, we have condition (i). In a similar fashion, $i\mathrm{Im}(z)\mathbf{I} - \mathbf{C}$ is the skew-Hermitian part of $z\mathbf{I} - \mathbf{A}$ and leads to condition (ii). Finally, condition (iii) of \cref{thm:arbitrary} requires $\mathrm{Re}(z)\mathbf{I}-\mathbf{B}$ and/or $i\mathrm{Im}(z)\mathbf{I} - \mathbf{C}$ to be indefinite, which implies condition (iii).
\end{proof}

\Cref{thm:arbitrary-value} can be specialized in the event that $\mathbf{A}$ is normal.

\begin{corollary}
Let $\mathbf{A}$ be normal. If $z$ is a {\prefix}value of $\mathbf{A}$, then the following three conditions hold.
\begin{enumerate}
\item[(i)] $\min \mathrm{Re}(\lambda(\mathbf{A})) \le \mathrm{Re}(z) \le \max \mathrm{Re}(\lambda(\mathbf{A}))$.
\item[(ii)] $\min \mathrm{Im}(\lambda(\mathbf{A})) \le \mathrm{Im}(z) \le \max \mathrm{Im}(\lambda(\mathbf{A}))$.
\item[(iii)] Conditions (i) and (ii) cannot both be strict equalities (similar to \cref{thm:arbitrary-value}).
\end{enumerate}
\label{thm:normal-value}
\end{corollary}
\begin{proof}
Let $\mathbf{B}$ and $\mathbf{C}$ be the Hermitian and skew-Hermitian parts of $\mathbf{A}$, respectively. This corollary follows directly from the facts that $\lambda(\mathbf{B})=\mathrm{Re}(\lambda(\mathbf{A}))$ and $\lambda(\mathbf{C})=i\mathrm{Im}(\lambda(\mathbf{A}))$ when $\mathbf{A}$ is normal.
\end{proof}

\Cref{thm:arbitrary-value} and \cref{thm:normal-value} provide bounds for {\prefix}values. In the event $\mathbf{A}$ is (skew-)Hermitian, we can also show the existence of {\prefix}values.

\begin{corollary}
Let $\mathbf{A}$ be Hermitian. $z$ is a {\prefix}value of $\mathbf{A}$ if and only if $z\in\mathbb{R}$ and $\min \lambda(\mathbf{A}) < z < \max \lambda(\mathbf{A})$. A similar statement using imaginary parts holds if $\mathbf{A}$ is skew-Hermitian instead.
\label{thm:hermitian-value}
\end{corollary}
\begin{proof}
($\Rightarrow$) Appealing to \cref{thm:normal-value}, $\min \lambda(\mathbf{A})\le \mathrm{Re}(z) \le \max\lambda(\mathbf{A})$ and $0\le \mathrm{Im}(z)\le 0$, with at most one equality holding. Clearly $\mathrm{Im}(z)=0$, implying that $z\in\mathbb{R}$ and $\min \lambda(\mathbf{A}) < z < \max\lambda(\mathbf{A})$. ($\Leftarrow$) $z\mathbf{I}-\mathbf{A}$ is indefinite when $\min \lambda(\mathbf{A}) < z < \max \lambda(\mathbf{A})$, which implies the existence of a nontrivial {\prefix}vector (\cref{thm:hermitian}).

Similar logic shows the corresponding statement when $\mathbf{A}$ is skew-Hermitian.
\end{proof}

\section{Finding {\Prefix}vectors and {\Prefix}values}
\label{sec:calculating}
The previous section detailed the existence of {\prefix}vectors and {\prefix}values of a matrix $\mathbf{A}$. We conclude this note with a brief discussion on calculating {\prefix}vectors and {\prefix}values. We hope this note inspires the development of additional tools.

In the most general case, a {\prefix}vector sits at the intersection of two quadric hypersurfaces. Equation \cref{eq:hermitian} shows that the {\prefix}vectors of the Hermitian part of $\mathbf{A}$ describe a quadric hypersurface when they are interpreted as displacement vectors from the origin. A similar result holds for the skew-Hermitian part. Because a {\prefix}vector of $\mathbf{A}$ must be a {\prefix}vector of both the Hermitian and skew-Hermitian parts of $\mathbf{A}$ by \cref{thm:arbitrary}, the {\prefix}vectors describe the intersection of two quadric hypersurfaces. This problem is generally discussed in \cite{thesis:reid-1972}. Furthermore, computer graphics has a longstanding interest in solving this problem \cite{levin-73-1979} --- usually in 3 or 4 dimensions --- and has developed algorithms that might be adaptable to the {\prefix}vector problem.

If $\mathbf{A}$ is normal, the situation simplifies because the quadric surfaces from $\mathbf{A}$'s Hermitian and skew-Hermitian parts share principal axes. In this case, the problem of finding {\prefix}vectors reduces to finding nonnegative solutions of a linear system.

\begin{theorem}
Let $\mathbf{A}$ be normal with eigenvalues $\lambda_1,\ldots,\lambda_n$. Every solution of the problem
\begin{subequations}
\label{eq:normal}
\begin{align}
0 & = \sum_{j=1}^n d_j \mathrm{Re}(\lambda_j), \label{eq:normal:real} \\
0 & = \sum_{j=1}^n d_j \mathrm{Im}(\lambda_j), \label{eq:normal:imag}
\end{align}
\end{subequations}
subject to $d_j\ge 0$ corresponds to a {\prefix}vector of $\mathbf{A}$.
\label{thm:normal}
\end{theorem}
\begin{proof}
Let $\varphi_j$ be a normalized eigenvector of $\mathbf{A}$ associated with eigenvalue $\lambda_j$ and let $\mathbf{B}$ and $\mathbf{C}$ be the Hermitian and skew-Hermitian parts of $\mathbf{A}$, respectively. Because $\mathbf{A}$ is normal, $\mathbf{A}$, $\mathbf{B}$, and $\mathbf{C}$ are simultaneously diagonalizable. Moreover, the spectrum of $\mathbf{B}$ is $\{\mathrm{Re}(\lambda_j)\}$; likewise, $\{i\mathrm{Im}(\lambda_j)\}$ is the spectrum of $\mathbf{C}$.

Suppose $x=\sum_j c_j \varphi_j$ is a {\prefix}vector of $\mathbf{A}$. From \cref{thm:arbitrary}, we have that $x$ is a {\prefix}vector of both $\mathbf{B}$ and $\mathbf{C}$. Mirroring logic from the proof of \cref{thm:hermitian},
\[
0 = \left< x, \mathbf{B}x \right> = \sum_{j=1}^n \left| c_j \right|^2 \mathrm{Re}(\lambda_j).
\]
Similarly,
\[
0 = \left< x, \mathbf{C}x \right> = \sum_{j=1}^n \left| c_j \right|^2 i \mathrm{Im}(\lambda_j) = \sum_{j=1}^n \left| c_j \right|^2 \mathrm{Im}(\lambda_j).
\]
Rewriting $d_j=\left|c_j\right|^2$ in the above equations gives a linear system where we want nonnegative solutions.
\end{proof}

\Cref{thm:normal} simplifies if $\mathbf{A}$ is Hermitian or skew-Hermitian because Eq.\ \cref{eq:normal:real} or \cref{eq:normal:imag}, respectively, is trivially satisfied. Furthermore, adding the equation $\sum_j d_j = 1$ to the system excludes the trivial solution $x=0$ without loss of generality.

One last result is somewhat of a corner case that comes from our research in nanotechnology \cite{sam-ang-053002-2017}. The conductance minima mentioned in \cref{sec:applications} are captured by a generalized eigenvalue problem that may have {\prefix}vector solutions. The key condition is that the skew-Hermitian part of $\mathbf{A}$ is negative semi-definite, and the present discussion readily generalizes to cases where either the Hermitian or skew-Hermitian part is positive or negative semi-definite.

\begin{theorem}
Let $\mathbf{B}$ and $\mathbf{C}$ be the Hermitian and skew-Hermitian parts of $\mathbf{A}$, respectively, with $\mathbf{C}$ either positive or negative semi-definite. Furthermore, let $\mathbf{P}$ be an orthogonal projector onto $\ker(\mathbf{C})$. If $x\in\ker(\mathbf{C})$ and $z$ are an eigenvector/eigenvalue pair of the generalized eigenvalue problem
\begin{equation}
\mathbf{PBP}x = z \mathbf{P}x,
\label{eq:geneig}
\end{equation}
then $z$ is a {\prefix}value of $\mathbf{A}$ and $x$ is a {\prefix}vector. A similar result holds if $\mathbf{B}$ is positive or negative semi-definite instead of $\mathbf{C}$.
\end{theorem}
\begin{proof}
Notice that Eq.\ \cref{eq:geneig} is Hermitian such that $z\in\mathbb{R}$. Then, from \cref{thm:arbitrary}, any {\prefix}vector $x$ of $z\mathbf{I}-\mathbf{A}$ must be a {\prefix}vector of $\mathbf{C}$, requiring $x\in\ker(\mathbf{C})$. Thus, $x=\mathbf{P}x$. Finally, $x$ and $z$ satisfying Eq.\ \cref{eq:geneig} implies
\begin{align*}
\left< x, (z\mathbf{I} - \mathbf{A}) x \right> & = \left< x, (z\mathbf{I} - \mathbf{B} - \mathbf{C}) x \right> \\
& = \left< x, (z\mathbf{I} - \mathbf{B}) x \right> \\
& = \left< \mathbf{P} x, (z\mathbf{I} - \mathbf{B}) \mathbf{P} x \right> \\
& = \left< x, \mathbf{P} (z\mathbf{I} - \mathbf{B}) \mathbf{P} x \right> \\
& = z \left< x, \mathbf{P} x \right> - \left< x, \mathbf{PBP} x \right> \\
& = 0.
\end{align*}
Hence, $x$ is a {\prefix}vector of $\mathbf{A}$ and $z$ is a {\prefix}value.
\end{proof}

\section*{Acknowledgments}
I thank Geoff Oxberry, Jay Bardhan, Andrew Mullhaupt, and David Keyes for helpful conversations.

\bibliographystyle{siamplain}
\bibliography{/Users/mgreuter/Documents/Publications/library}

\end{document}